\documentclass[11pt,english]{article}
\usepackage[utf8]{inputenc}
\usepackage[a4paper]{geometry}
\usepackage{babel}
\usepackage{mathtools}
\usepackage{dsfont}
\usepackage{amsmath}
\usepackage{amsthm}
\usepackage{amssymb}
\usepackage{stmaryrd}
\usepackage[bookmarks=true,bookmarksnumbered=false,bookmarksopen=false,
 breaklinks=false,pdfborder={0 0 1},backref=false,colorlinks=false]
 {hyperref}
\hypersetup{
 linkcolor=blue}

\makeatletter
\numberwithin{equation}{section}
\numberwithin{figure}{section}
\theoremstyle{plain}
\newtheorem{thm}{\protect\theoremname}[section]
\theoremstyle{plain}
\newtheorem{lem}[thm]{\protect\lemmaname}

\usepackage{ dsfont }

\newcommand{\N}{\mathbb{N}}

\newcommand{\df}{\mathrm{d}}

\allowdisplaybreaks

\makeatother

\providecommand{\lemmaname}{Lemma}
\providecommand{\theoremname}{Theorem}

\begin{document}
\global\long\def\df{\mathrm{def}}%
\global\long\def\eqdf{\stackrel{\df}{=}}%
\global\long\def\ep{\varepsilon}%
\global\long\def\ind{\mathds{1}}%
\global\long\def\cl{\mathrm{cl}}%
\global\long\def\N{\mathbb{N}}%
\global\long\def\C{\mathbb{C}}%

\title{Spectral gap with polynomial rate for random covering surfaces}
\author{Will Hide, Davide Macera and Joe Thomas}
\maketitle
\begin{abstract}
In this note we show that the recent work of Magee, Puder and van
Handel \cite{Ma.Pu.vH2025} can be applied to obtain an optimal spectral
gap result with polynomial error rate for uniformly random covers
of closed hyperbolic surfaces. 

Let $X$ be a closed hyperbolic surface. We show there exists $b,c>0$
such that a uniformly random degree-$n$ cover $X_{n}$ of $X$ has
no new Laplacian eigenvalues below $\frac{1}{4}-cn^{-b}$ with probability
tending to $1$ as $n\to\infty$.

{\footnotesize\tableofcontents{}}{\footnotesize\par}
\end{abstract}

\section{Introduction}

Let $X$ be a closed and connected hyperbolic surface. The (positive)
Laplacian on $L^{2}(X)$ has discrete spectrum in $[0,\infty)$ with
a simple eigenvalue at zero given by constant functions. In this article
we investigate the size of the first \textit{new} eigenvalue $\lambda_{1}^{\text{new}}\left(X_{n}\right)$
of the Laplacian on a random degree-$n$ cover $X_{n}$ of $X$. The
\textit{new} eigenvalues on the cover are precisely those arising
from the restriction of the Laplacian on the cover $X_{n}$ to the
subspace of $L^{2}(X_{n})$ orthogonal to the span of lifts of functions
in $L^{2}(X)$. 

We will consider \emph{uniformly} random covers of $X$ (with labelled
fiber). That is, for any $n\in\mathbb{N}$, there are only finitely
many degree-$n$ covers of $X$ and we sample them with the uniform
probability measure. Equivalently, the covers are in bijection with
homomorphisms $\varphi_{n}:\Gamma\to S_{n}$ where $\Gamma=\pi_{1}(X)$
and $S_{n}$ is the symmetric group on $n$ objects. Thus the random
model is a uniformly random choice of $\mathrm{Hom}(\Gamma,S_{n})$
(see Section \ref{subsec:Function-spaces} for more details). A theorem
of Liebeck and Shalev \cite[Theorem 1.12]{Li.Sh2024} ensures that
with probability tending to $1$ as $n\to\infty$ a cover is connected
whereupon the first new eigenvalue is strictly positive. 

\subsection*{Spectral gap}

An asymptotically optimal lower bound for $\lambda_{1}^{\mathrm{new}}(X_{n})$
is $\frac{1}{4}$ which is the bottom of the Laplacian spectrum on
the hyperbolic plane \cite{Huber}. Let $[n]=\left\{ 1,\ldots,n\right\} $
and let $\mathrm{std}:S_{n}\to V_{n}^{0}$ denote the standard representation
acting by permutation matrices on the irreducible subspace $V_{n}^{0}\subset\ell^{2}([n])$
of functions orthogonal to constant functions. In a recent breakthrough
of Magee, Puder and van Handel \cite{Ma.Pu.vH2025}, it has been proven
that for uniformly random $\varphi_{n}\in\mathrm{Hom}(\Gamma,S_{n})$
the representations $\mathrm{std}_{n}\circ\varphi_{n}$ strongly converge
in probability to the left regular representation $\lambda_{\Gamma}:\Gamma\to\ell^{2}(\Gamma)$
(see Section \ref{subsec:Strong-convergence}). As a consequence (using
for example the methods of \cite{Hi.Ma2023,Lo.Ma2023}), one is able
to deduce that for any $\varepsilon>0$, with probability tending
to 1 as $n\to\infty$, $\lambda_{1}^{\mathrm{new}}(X_{n})>\frac{1}{4}-\varepsilon$;
and in fact, $\varepsilon$ can be taken to be of the order $\frac{\log\log(n)}{\log(n)}$
\cite{Hi2023}. 

The aim of this short note is to show that the methods of \cite{Ma.Pu.vH2025}
can be used to obtain a polynomial decay rate for $\varepsilon$,
vastly improving the order of the error rates for the spectral gap
on hyperbolic surfaces in any of the random models to date. We review
the previous literature in this regard and the expected optimal rate
in Section \ref{subsec:Previous-works}.
\begin{thm}
\label{thm:main-thm}Let $X$ be a closed hyperbolic surface. There
exists $b,c>0$ depending only on the genus of $X$ such that a uniformly
random degree-$n$ cover $X_{n}$ of $X$ has 
\[
\lambda_{1}^{\textup{new}}\left(X_{n}\right)\geqslant\frac{1}{4}-cn^{-b},
\]
with probability tending to $1$ as $n\to\infty$. 
\end{thm}

\subsection{Previous works}

\label{subsec:Previous-works}

For any fixed base surface $X$, uniformly random degree-$n$ covers
$X_{n}$ Benjamini-Schramm converge to the hyperbolic plane, their
universal cover, with probability tending to $1$ as $n\to\infty$
\cite{MPasympcover}. A consequence of this is that, after rescaling
by $\frac{1}{n\mathrm{Vol}(X)}$, the empirical spectral measure of
the Laplacian weakly converges to the empirical spectral measure of
the hyperbolic plane whose density with respect to the Lebesgue measure
is given by

\[
\frac{1}{4\pi}\mathds{1}_{[\frac{1}{4},\infty)}(\lambda)\tanh\left(\pi\sqrt{\lambda-\frac{1}{4}}\right)\mathrm{d}\lambda.
\]
For $\lambda$ near to $\frac{1}{4}$, the density has the order $\sqrt{\lambda-\frac{1}{4}}$,
and thus we expect that random $\lambda_{1}^{\mathrm{new}}(X_{n})$
should typically fluctuate in an interval of order $n^{-\frac{2}{3}+\varepsilon}$
around $\frac{1}{4}$. In fact, it is expected from the Bohigas, Giannoni
and Schmidt conjecture \cite{Bo.Gi.Sc1984} that due to the time-reversal
symmetry and chaotic nature of the geodesic flow on hyperbolic surfaces,
the spectral statistics of the Laplacian should exhibit fluctuation
properties akin to the Gaussian Orthogonal Ensemble. For this random
matrix ensemble, the limiting behavior of the largest eigenvalue (with
suitable normalization) is given by the Tracy-Widom distribution with
$\beta=1$ \cite{Tr.Wi1996}. It could thus be expected that for uniform
random covers, $\lambda_{1}^{\mathrm{new}}(X_{n})$ also fluctuates
similarly in the $n\to\infty$ limit. 

A geometrically closely related setting to the surfaces considered
here, and hence a potential indicator for what one can expect to see
for surfaces, is that of regular graphs. In a recent breakthrough
of Huang, McKenzie and Yau \cite{Hu.Mc.Ya2024}, it is proven that
for uniformly random $d$-regular graphs on $n$ vertices, the largest
non-trivial eigenvalue does indeed fluctuate around its expected position
according to the semicircle law up to the optimal $n^{-\frac{2}{3}+\varepsilon}$
scale. Moreover, it is shown to be distributed according to the Tracy-Widom
distribution with $\beta=1$ (and similarly for the smallest eigenvalue)
as $n\to\infty$. As a consequence, they conclude that with probability
approximately $69\%$, a $d$-regular graph is Ramanujan - all non-trivial
eigenvalues lie within the bulk. 

For hyperbolic surfaces, it has only recently been established that
they typically exhibit an almost optimal spectral gap, that is, $\lambda_{1}^{\mathrm{new}}(X_{n})>\frac{1}{4}-\ep$
as $n\to\infty$. This was first shown for random covers of finite-area
non-compact hyperbolic surfaces by the first named author and Magee
in \cite{Hi.Ma2023}. Via a compactifiction procedure \cite{BBD},
this proved the existence of a sequence of closed hyperbolic surfaces
with genera $\to\infty$ and with $\lambda_{1}\to\frac{1}{4}$. An
alternative proof, allowing one to take a tower of covers of a fixed
closed surface, was given in \cite{Lo.Ma2023}. Together with quantitative
strong convergence results of Bordeanave and Collins \cite{Bo.Co2023},
in \cite{Hi2023} the arguments in \cite{Hi.Ma2023} were used to
give an explicit rate for $\varepsilon$ of the order $\frac{(\log\log\log n)^{2}}{\log\log n}$. 

For random closed hyperbolic surfaces there has been some spectacular
recent progress. In the Weil-Petersson random model for surfaces of
large genus, Anantharaman and Monk \cite{An.Mo2025} have recently
obtained the spectral gap $\frac{1}{4}-\ep$ (improving on prior results
\cite{MirzakhaniRandom,Wu.Xu2022,Li.Wr2024,An.Mo2023} in this model).
Moreover, Magee, Puder and van Handel \cite{Ma.Pu.vH2025} have also
obtained the spectral gap $\frac{1}{4}-\varepsilon$ in the uniform
random covering model (see also the previous work \cite{Ma.Na.Pu2022}). 

Theorem \ref{thm:main-thm} is the first result in the literature
where a polynomial error rate has been established for the spectral
gap of random surfaces, marking significant improvement over previous
rates. This has been made possible due to the recent breakthrough
on the strong convergence of surface groups in \cite{Ma.Pu.vH2025}
which significantly builds upon the new methodology to study strong
convergence developed in the remarkable works \cite{Ch.Ga.Tr.va2024,Ch.Ga.Ha2024}.
In the latter references, the strong convergence results were applied
to give a new proof of polynomial error rates for the size of the
largest non-trivial eigenvalue on random $2d$-regular graphs.

\section{Preliminaries}

We will access the spectral gap of a cover using the functional calculus,
specifically, as the operator norm of the Selberg transform of an
appropriately chosen kernel function. This operator norm is then compared
to the analogous operator norm on $L^{2}(\mathbb{H})$ which one can
calculate directly. To undertake the comparison, we will make use
of the recent breakthrough result of Magee, Puder and van Handel \cite{Ma.Pu.vH2025}
proving that in probability, for uniformly random permutation representations
$\varphi_{n}$, the representations $\mathrm{std}_{n}\circ\varphi_{n}$
strongly converge to the regular representation on $\Gamma_{g}$ as
$n\to\infty$. In this section, we outline the necessary background
and results that we will utilize in our proof.

\subsection{Selberg transform}

On a compact hyperbolic surface $X=\Gamma_{g}\backslash\mathbb{H}$,
the spectrum of the (positive) Laplacian is discrete and consists
of eigenvalues $0=\lambda_{0}<\lambda_{1}\leq\lambda_{2}\leq\ldots$
with $\lambda_{j}\to\infty$ as $j\to\infty$. The space $L^{2}(X)$
has an orthonormal basis of Laplacian eigenfunctions $\left\{ \psi_{j}\right\} _{j\geq0}$
with $\Delta\psi_{j}=\lambda_{j}\psi_{j}$.

Given a smooth and compactly supported function $k:[0,\infty)\to\mathbb{R}$
one can construct (by an abuse of notation) a kernel function $k:\mathbb{H}\times\mathbb{H}\to\mathbb{R}$
by
\[
k(z,w)=k(d(z,w)),
\]
with associated integral operator $P_{k}$ given by
\[
P_{k}f(z)=\int_{\mathbb{H}}k(z,w)f(w)\mathrm{d}\mu_{\mathbb{H}}(w).
\]
Here $\mu_{\mathbb{H}}$ is the standard area measure on the hyperbolic
plane. We define $P_{k}:L^{2}(X)\to L^{2}(X)$ by

\[
P_{k}f(z)=\int_{\mathcal{F}}\sum_{\gamma\in\Gamma_{g}}k(z,\gamma w)f(w)\mathrm{d}\mu_{\mathbb{H}}(w),
\]
where $\mathcal{F}$ is a Dirichlet fundamental domain for $X$. The
operator $P_{k}$ preserves the eigenspaces of the Laplacian on $L^{2}(X)$
and the eigenvalues are transformed via the Selberg transform of the
kernel $k$. We denote the corresponding Selberg transform of $k$
by $h$, which is defined by (see for example \cite{Be2016})
\[
h(r)=\sqrt{2}\int_{-\infty}^{\infty}e^{iru}\int_{|u|}^{\infty}\frac{k(\rho)\sinh(\rho)}{\sqrt{\cosh(\rho)-\cosh(u)}}\mathrm{d\rho\mathrm{d}u}.
\]
Then, parameterizing the eigenvalues by $\lambda_{j}=\frac{1}{4}+r_{j}^{2}$
so that $r_{j}\in[0,\infty)\cup[0,\frac{1}{2}]i$ one has

\[
P_{k}\psi_{j}=h(r_{j})\psi_{j}.
\]
Since the eigenfunctions form an orthonormal basis of $L^{2}(X),$and
the operator $P_{k}$ is self-adjoint (as $k$ is real-valued) we
have
\[
\|P_{k}\|_{L^{2}(X)\to L^{2}(X)}=\sup_{r\in[0,\infty)\cup[0,\frac{1}{2}]i}|h(r)|.
\]
Similarly, on the hyperbolic plane, since $L^{2}(\mathbb{H})$ has
a generalized eigenbasis of $C^{\infty}$ eigenfunctions, the Borel
functional calculus allows for the extension of $P_{k}$ to an operator
from $C_{c}^{\infty}(\mathbb{H})\to C_{c}^{\infty}(\mathbb{H})$ to
a possibly unbounded and self-adjoint operator acting $L^{2}(\mathbb{H})\to L^{2}(\mathbb{H})$.
Since the $L^{2}$ spectrum of the Laplacian on $\mathbb{H}$ is equal
to $[\frac{1}{4},\infty)$, we obtain

\[
\|P_{k}\|_{L^{2}(\mathbb{H})\to L^{2}(\mathbb{H})}=\sup_{r\in[0,\infty)}|h(r)|.
\]

\subsection{Function spaces}

\label{subsec:Function-spaces}

In this section we describe the various functions spaces that will
be of interest to us. As before, let $X$ be a closed surface of genus
$g$ realized as a quotient of the hyperbolic plane by a surface group
$\Gamma_{g}$. Let $\varphi_{n}:\Gamma_{g}\to S_{n}$ be a permutation
representation and $X_{n}$ be the corresponding degree-$n$ covering
surface. That is,

\[
X_{n}=\Gamma_{g}\backslash_{\varphi}(\mathbb{H}\times[n]),
\]
where $\Gamma_{g}$ acts by isometries on $\mathbb{H}$ and by permutations
through $\varphi_{n}$ on $[n]$:

\[
\gamma\cdot(z,i)=(\gamma z,\varphi_{n}(\gamma)(i)).
\]
The space $L^{2}(X_{n})$ decomposes as
\[
L^{2}(X_{n})\cong L_{\mathrm{new}}^{2}(X_{n})\oplus L^{2}(X),
\]
where $L^{2}(X)$ is the lift of functions from the base to the cover
and $L_{\mathrm{new}}^{2}(X_{n})$ is the corresponding orthogonal
complement. 

Now let $V_{n}=\ell^{2}([n])$ and $V_{n}^{0}\subseteq V_{n}$ be
the subspace of functions orthogonal to constants, that is, with mean
zero. The symmetric group $S_{n}$ has a natural action on $V_{n}$
by the standard representations $\mathrm{std}_{n}$ of permutation
matrices; the subspace $V_{n}^{0}$ is an $(n-1)$-dimensional irreducible
subrepresentation. By composing $\varphi_{n}$ with $\mathrm{std_{n}}$
we obtain a representation on $V_{n}^{0}$, that is, we denote by
\[
\rho_{\varphi_{n}}:\Gamma_{g}\to\mathrm{End}(V_{n}^{0}),
\]
the representation $\rho_{\varphi_{n}}=\mathrm{std}_{n}\circ\varphi_{n}$. 

The importance of the subspace $V_{n}^{0}$ is that it allows access
to $L_{\mathrm{new}}^{2}(X_{n})$. Precisely, if we let $C^{\infty}(\mathbb{H};V_{n}^{0})$
be the collection of smooth $V_{n}^{0}$-valued functions on $\mathbb{H}$,
then there is an isometric linear isomorphism between 
\[
C^{\infty}(X_{n})\cap L_{\mathrm{new}}^{2}(X_{n}),
\]
and the subspace $C_{\varphi}^{\infty}(\mathbb{H};V_{n}^{0})$ of
$C^{\infty}(\mathbb{H};V_{n}^{0})$ consisting of functions satisfying
the automorphy condition
\[
f(\gamma z)=\rho_{\varphi_{n}}(\gamma)f(z),\qquad\text{for all }\gamma\in\Gamma_{g},z\in\mathbb{H},
\]
with finite norm

\[
\|f\|_{L^{2}(\mathcal{F})}^{2}\eqdf\int_{\mathcal{F}}\|f(z)\|_{V_{n}^{0}}^{2}\mathrm{d}\mu_{\mathbb{H}}(z)<\infty.
\]
Recall here that $\mathcal{F}$ is a fixed Dirichlet fundamental domain
for $X$ and $\mu_{\mathbb{H}}$ is the standard hyperbolic area measure.
The completion of $C_{\varphi}^{\infty}(\mathbb{H};V_{n}^{0})$ with
respect to the norm $\|\cdot\|_{L^{2}(\mathcal{F})}$ is denoted by
$L_{\varphi}^{2}(\mathbb{H};V_{n}^{0})$ and the isomorphism above
extends to one between $L_{\mathrm{new}}^{2}(X_{n})$ and $L_{\varphi}^{2}(\mathbb{H};V_{n}^{0})$.
The space $L_{\varphi}^{2}(\mathbb{H};V_{n}^{0})$ is also isometrically
isomorphic to $L^{2}(\mathcal{F})\otimes V_{n}^{0}$, which can be
realized via the map 
\begin{equation}
f\in L_{\varphi}^{2}(\mathbb{H};V_{n}^{0})\mapsto\sum_{i=1}^{n-1}\left\langle f|_{\mathcal{F}}(\cdot),e_{i}\right\rangle _{V_{n}^{0}}\otimes e_{i},\label{eq:isomorphism}
\end{equation}
for some basis $\left\{ e_{i}\right\} _{i=1}^{n-1}$ of $V_{n}^{0}$.
We will frequently pass between these isomorphic spaces in the later
proof. There is also an isometric isomorphism between $L^{2}(\mathcal{F})\otimes\ell^{2}(\Gamma_{g})$
and $L^{2}(\mathbb{H})$ given by $f\otimes\delta_{\gamma}\mapsto f\circ\gamma^{-1}$.

\subsection{Strong convergence}

\label{subsec:Strong-convergence}

We continue with the notation from the previous section. The recent
breakthrough of Magee, Puder and van Handel \cite{Ma.Pu.vH2025} proved
that when $\varphi_{n}$ is sampled uniformly at random, the corresponding
representations $\rho_{\varphi_{n}}$ strongly converge in probability
to the regular representation $\lambda_{\Gamma_{g}}$ on $\Gamma_{g}$
as $n\to\infty.$ Recall that $\lambda_{\Gamma_{g}}:\Gamma_{g}\to U(\ell^{2}(\Gamma_{g}))$
is defined by 
\[
\lambda_{\Gamma_{g}}(\gamma)[f](g)=f(\gamma^{-1}g),
\]
for all $\gamma\in\Gamma_{g}$ and $f\in\ell^{2}(\Gamma_{g})$.
\begin{thm}[{\cite[Theorem 1.1]{Ma.Pu.vH2025}}]
For all $x\in\mathbb{C}[\Gamma_{g}]$, for $\varphi\in\mathrm{Hom}(\Gamma_{g},S_{n})$
sampled uniformly at random as $n\to\infty,$ one has in probability
that
\[
\|\rho_{\varphi_{n}}(x)\|_{\mathrm{End}(V_{n}^{0})\to\mathrm{End}(V_{n}^{0})}\to\|\lambda_{\Gamma_{g}}\|_{\ell^{2}(\Gamma_{g})\to\ell^{2}(\Gamma_{g})}.
\]
\end{thm}

In fact, a stronger result is proven which deals with not just elements
of the group algebra, but general matrix coefficient polynomials and
we will make use of an effective version of this result here.
\begin{thm}[{\cite[Theorem 6.1]{Ma.Pu.vH2025}}]
\label{thm:MPvH}For any self-adjoint $x\in M_{d}(\mathbb{C})\otimes\mathbb{C}[\Gamma_{g}]$,
$n\geq1$, and $\varepsilon>0$, we have
\[
\mathbb{P}\left(\|(\mathrm{id}\otimes\rho_{\varphi_{n}})(x)\|\geq(1+\varepsilon)\|(\mathrm{id}\otimes\lambda_{\Gamma_{g}})(x)\|\right)\leq\frac{cd}{n\varepsilon^{b}},
\]
where $b$ is a constant depending only on $g$, and $c$ depends
on both $g$ and the word length of $x$. The first norm is the operator
norm on $M_{d}(\mathbb{C})\otimes V_{n}^{0}$ and the second is the
operator norm on $M_{d}(\mathbb{C})\otimes\ell^{2}(\Gamma_{g})$. 
\end{thm}

\section{Proof of Theorem \ref{thm:main-thm}}

The integral operator that we will use to study the spectral gap will
be a geometric ball cutoff which has been used in previous works on
random hyperbolic surfaces studying quantum ergodicity for eigenfunctions,
$L^{p}$-norms, and bass note spectra \cite{Gi.Le.Sa.Th2021,Le.Sa2024,Ma2024}.
Let $k_{t}:\mathbb{H}\times\mathbb{H}\to\mathbb{R}$ be defined for
all $t\geq0$ by 
\[
k_{t}(z,w)=\mathds{1}_{d(z,w)\leq t},
\]
so that this kernel correponds to just the indicator function on the
interval $[0,t]$. Even though $k_{t}$ is not smooth, its Selberg
transform exists and can be computed as 

\[
h_{t}(r)=4\sqrt{2}\int_{0}^{t}\cos(ru)\sqrt{\cosh(t)-\cosh(u)}\mathrm{d}u.
\]
We will let $P_{k_{t}}$ denote the associated integral operator induced
by $k_{t}$ where it will be clear from context whether we are regarding
it acting on the appropriate spaces $L^{2}(X),L_{\mathrm{new}}^{2}(X_{n})$,
or $L^{2}(\mathbb{H})$.

We first show that the operator norm of $P_{k_{t}}$ does indeed capture
desirable information about $\lambda_{1}^{\mathrm{new}}(X_{n})$.
\begin{lem}
\label{lem:op-norm}Suppose that $\lambda_{1}^{\mathrm{new}}(X_{n})\leq\frac{1}{4}$,
then
\[
\|P_{k_{t}}\|_{L_{\mathrm{new}}^{2}(X_{n})\to L_{\mathrm{new}}^{2}(X_{n})}=h_{t}\left(i\sqrt{\frac{1}{4}-\lambda_{1}^{\mathrm{new}}(X_{n})}\right).
\]
\end{lem}

\begin{proof}
The operator norm is equal to the supremum of $|h_{t}(r)|$ as $r$
runs over all spectral parameters corresponding to new eigenvalues
$\lambda=\frac{1}{4}+r^{2}$ on the cover. For any eigenvalue larger
than $\frac{1}{4}$ we have 
\begin{equation}
|h_{t}(r)|\leq4\sqrt{2}\int_{0}^{t}\sqrt{\cosh(t)-\cosh(u)}\mathrm{d}u.\label{eq:big-evalue-1}
\end{equation}
For new eigenvalues below $\frac{1}{4}$, the spectral parameter $r$
is equal to $ai$ for some $a\in[0,\frac{1}{2})$. In this case,

\[
|h_{t}(ai)|=4\sqrt{2}\int_{0}^{t}\cosh(au)\sqrt{\cosh(t)-\cosh(u)}\mathrm{d}u,
\]
which always dominates (\ref{eq:big-evalue-1}) and moreover, this
integral is maximized when $a$ is maximal, that is, when $a=\sqrt{\frac{1}{4}-\lambda_{1}^{\mathrm{new}}(X_{n})}$. 
\end{proof}
Recall that $L_{\mathrm{new}}^{2}(X_{n})\cong L^{2}(\mathcal{F})\otimes V_{n}^{0}$,
and under this isomorphism the integral operator $P_{k_{t}}$ is conjugated
to
\begin{equation}
\sum_{\gamma\in\Gamma_{g}}a_{\gamma,t}\otimes\rho_{n}(\gamma^{-1}),\label{eq:conj-operator}
\end{equation}
where $a_{\gamma,t}:L^{2}(\mathcal{F})\to L^{2}(\mathcal{F})$ is
given by 
\[
(a_{\gamma,t}f)(z)=\int_{\mathcal{F}}k_{t}(d(z,\gamma w))f(w)\mathrm{d}\mu_{\mathbb{H}}(w).
\]
The compact support of $k_{t}$ means that the number of $\gamma\in\Gamma$
for which $a_{\gamma,t}\neq0$ is bounded only in terms of $t$ and
the base surface $X$. 
\begin{lem}
\cite[Lemma 5.1]{Hi2023}\label{lem:support}The map $\gamma\mapsto a_{\gamma,t}$
is supported on a set $S(t)\subseteq\Gamma_{g}$ with 
\[
|S(t)|\leq Ce^{2t},
\]
for some absolute constant $C>0$ depending only on $X$. Moreover,
the elements in the set $S(t)$ correspond to geodesics on $X$ of
length at most $2(\mathrm{diam}(\mathcal{F})+t+1)$.
\end{lem}

\begin{proof}
We only sketch the proof, the full details can be found in \cite[Lemma 5.1]{Hi2023}.
Recall that $\mathcal{F}\subseteq\mathbb{H}$ is a Dirichlet fundamental
domain $X$ and suppose that it is based at a point $x$. The kernel
$k_{t}(d(z,\gamma w))$ is non-zero only when $d(z,\gamma w)\leq t+1$
for $z,w\in\mathcal{F}$. If $\gamma\in\Gamma$ is such that $d(z,\gamma w)\leq t+1$
then
\[
d(x,\gamma x)\leq d(x,\gamma w)+d(\gamma w,\gamma x)\leq2d(x,\gamma w)\leq2(d(x,z)+d(z,\gamma w))\leq2(\mathrm{diam}(\mathcal{F})+t+1).
\]
It follows that $\gamma\to a_{\gamma,t}$ is only non-zero on $\left\{ \gamma\in\Gamma:d(x,\gamma x)\leq2(\mathrm{diam}(\mathcal{F})+t+1)\right\} $
which by a lattice point count has size bounded by $e^{2t}$ up to
a multiplicative constant depending only on the base surface $X$.
The latter claim follows from the fact that the length of the closed
geodesic corresponding to $\gamma$ is bounded above by $d(x,\gamma x)$.
\end{proof}
The operator (\ref{eq:conj-operator}) is almost in the form of the
polynomials considered in Theorem \ref{thm:MPvH} except that the
operators $a_{\gamma,t}$ are not finite rank. We will thus approximate
them by finite-rank operators with a quantitative error rate. 
\begin{lem}
\label{lem:finite-rank-approx}There exists a constant $C>0$ such
that for any $r\in\mathbb{N}$ there is a finite dimensional subspace
$W\subseteq L^{2}(\mathcal{F})$ of rank at most $r|S(t)|$ and operators
$b_{\gamma,t}^{(r)}:W\to W$ for every $\gamma\in S(t)$ for which
\[
\|a_{\gamma,t}-b_{\gamma,t}^{(r)}\|_{L^{2}(\mathcal{F})\to L^{2}(\mathcal{F})}\leq\frac{\|a_{\gamma,t}\|_{\mathrm{HS}}}{\sqrt{r}}\leq\frac{Ce^{t}}{\sqrt{r}}.
\]
\end{lem}

\begin{proof}
The proof is identical to \cite[Lemma 5.2]{Hi2023} and so we only
sketch the main ideas. For each $\gamma\in S(t)$, the operator $a_{\gamma,t}$
is compact and so it has a singular value decomposition with a decreasing
sequence of singular values $\left\{ s_{j}\right\} _{j\geq1}$. As
$a_{\gamma,t}$ is a Hilbert-Schmidt operator, we have $\sum_{j=1}^{\infty}s_{j}^{2}=\|a_{\gamma,t}\|_{\mathrm{HS}}^{2}$.

Let 
\[
a_{\gamma,t}=\sum_{j=1}^{\infty}s_{j}\left\langle \cdot,e_{i}\right\rangle f_{i}
\]
for some orthonormal systems $\left\{ e_{i}\right\} _{i\in\mathbb{N}}$
and $\left\{ f_{i}\right\} _{i\in\mathbb{N}}$ in $L^{2}(\mathcal{F})$.
Then for $r\in\mathbb{N}$, we define
\[
b_{\gamma,t}^{(r)}=\sum_{j=1}^{r}s_{j}\left\langle \cdot,e_{i}\right\rangle f_{i},
\]
so that $b_{\gamma,t}^{(r)}$ is a finite rank operator acting on
a subspace $W_{\gamma}\subseteq L^{2}(\mathcal{F})$ of dimension
at most $r$. We let $W=\bigcup_{\gamma\in S(t)}W_{\gamma}$ which
has size at most $r|S(t)|$ so that all $b_{\gamma,t}^{(r)}$ act
on the space simultaneously. Since the $s_{j}$ are decreasing, we
have
\[
(r+1)s_{r+1}^{2}\leq\sum_{j=1}^{r+1}s_{j}^{2}\leq\|a_{\gamma,t}\|_{\mathrm{HS}}^{2}.
\]
Moreover, since $k_{t}$ is supported in a ball of radius $t$, we
have $\|a_{\gamma,t}\|_{\mathrm{HS}}=\|k_{t}\|_{L^{2}(\mathcal{F}\times\mathcal{F})}\leq Ce^{t}$.
But then,
\[
\|a_{\gamma,t}-b_{\gamma,t}^{(r)}\|_{L^{2}(\mathcal{F})\to L^{2}(\mathcal{F})}\leq s_{r+1}\leq\frac{\|a_{\gamma,t}\|_{\mathrm{HS}}}{\sqrt{r}}.
\]
 
\end{proof}
We now prove Theorem \ref{thm:main-thm}.
\begin{proof}[Proof of Theorem \ref{thm:main-thm}]
We pick $t=1$ in the kernel $k_{t}$. Then,
\[
\|P_{k_{1}}\|_{L_{\mathrm{new}}^{2}(X_{n})\to L_{\mathrm{new}}^{2}(X_{n})}=h_{1}\left(i\sqrt{\frac{1}{4}-\lambda_{1}^{\mathrm{new}}(X_{n})}\right)
\]
by Lemma \ref{lem:op-norm}. Thus,
\begin{align*}
\|P_{k_{1}}\|_{L_{\mathrm{new}}^{2}(X_{n})\to L_{\mathrm{new}}^{2}(X_{n})} & =\left\Vert \sum_{\gamma\in S(1)}a_{\gamma,1}\otimes\rho_{n}(\gamma^{-1})\right\Vert _{L^{2}(\mathcal{F})\otimes V_{n}^{0}}\\
 & \leq\left\Vert \sum_{\gamma\in S(1)}b_{\gamma,1}^{(r)}\otimes\rho_{n}(\gamma^{-1})\right\Vert _{W\otimes V_{n}^{0}}+\sum_{\gamma\in S(1)}\|a_{\gamma,1}-b_{\gamma,1}^{(r)}\|_{L^{2}(\mathcal{F})\to L^{2}(\mathcal{F})}.
\end{align*}
The second term on the right-hand side is bounded by $Ar^{-\frac{1}{2}}$
for some constant $A>0$ depending only on the base surface $X$ by
Lemmas \ref{lem:support} and \ref{lem:finite-rank-approx}. This
means that
\begin{equation}
\sum_{\gamma\in S(1)}b_{\gamma,1}^{(r)}\otimes\gamma^{-1}\in M_{d}(\mathbb{C})\otimes\mathbb{C}[\Gamma_{g}]\label{eq:poly}
\end{equation}
with $d\leq r|S(1)|$ by Lemma \ref{lem:finite-rank-approx}. By Lemma
\ref{lem:support}, the elements in $S(1)$ have geodesic length uniformly
bounded by a constant dependent only upon $X$ and so since by the
Švarc-Milnor lemma \cite[Ch. 1 Proposition 8.19]{Br.Ha1999} the word
length and geodesic length are quasi-isometric, the word length of
any $\gamma\in S(1)$ is also uniformly bounded by a constant dependent
only on $X$.

To apply Theorem \ref{thm:MPvH}, we require that (\ref{eq:poly})
is self-adjoint which can be guaranteed after possibly replacing $M_{d}(\mathbb{C})$
with $M_{d}(\mathbb{C})\otimes M_{2}(\mathbb{C})$ as in \cite[Proof of Theorem 1.1]{Bo.Co2023}.
In particular we may modify the polynomial to
\begin{equation}
\sum_{\gamma\in S(1)}\tilde{b}_{\gamma,1}^{(r)}\otimes\gamma^{-1}\in M_{d}(\mathbb{C})\otimes M_{2}(\mathbb{C})\otimes\mathbb{C}[\Gamma_{g}],\label{eq:finite-rank-self-adjoint}
\end{equation}
where
\[
\tilde{b}_{\gamma,1}^{(r)}=\left(\begin{array}{cc}
0 & b_{\gamma,1}^{(r)}\\
\left(b_{\gamma,1}^{(r)}\right)^{*} & 0
\end{array}\right),
\]
which is now a self-adjoint polynomial whose matrix coefficients are
of dimension at most $2r|S(1)|$ and for any unitary representation
$\sigma:\Gamma_{g}\to V$, 
\[
\left\Vert \sum_{\gamma\in S(1)}b_{\gamma,1}^{(r)}\otimes\sigma(\gamma^{-1})\right\Vert _{M_{d}(\mathbb{C})\otimes V}=\left\Vert \sum_{\gamma\in S(1)}\tilde{b}_{\gamma,1}^{(r)}\otimes\sigma(\gamma^{-1})\right\Vert _{M_{d}(\mathbb{C})\otimes M_{2}(\mathbb{C})\otimes V}.
\]
We now apply Theorem \ref{thm:MPvH} so that with probability $1-O_{X}\left(\frac{1}{\left(\log(n)\right)^{\frac{1}{b}}}\right)$
we have
\begin{align*}
\left\Vert \sum_{\gamma\in S(1)}b_{\gamma,1}^{(r)}\otimes\rho_{n}(\gamma^{-1})\right\Vert _{M_{d}(\mathbb{C})\otimes V_{n}^{0}}\\
 & \hspace*{-1.5cm}\leq\left\Vert \sum_{\gamma\in S(1)}\tilde{b}_{\gamma,1}^{(r)}\otimes\lambda_{\Gamma_{g}}(\gamma^{-1})\right\Vert _{M_{d}(\mathbb{C})\otimes M_{2}(\mathbb{C})\otimes\ell^{2}(\Gamma_{g})}\left(1+\left(\frac{r\log(n)}{n}\right)^{\frac{1}{b}}\right)\\
 & \hspace*{-1.5cm}=\left\Vert \sum_{\gamma\in S(1)}b_{\gamma,1}^{(r)}\otimes\lambda_{\Gamma_{g}}(\gamma^{-1})\right\Vert _{M_{d}(\mathbb{C})\otimes\ell^{2}(\Gamma_{g})}\left(1+\left(\frac{r\log(n)}{n}\right)^{\frac{1}{b}}\right).
\end{align*}
A second application of Lemma \ref{lem:finite-rank-approx} gives
\begin{align*}
\left\Vert \sum_{\gamma\in S(1)}b_{\gamma,1}^{(r)}\otimes\rho_{n}(\gamma^{-1})\right\Vert _{M_{d}(\mathbb{C})\otimes V_{n}^{0}}\\
 & \hspace*{-2cm}\leq\left(\left\Vert \sum_{\gamma\in S(1)}a_{\gamma,1}\otimes\lambda_{\Gamma_{g}}(\gamma^{-1})\right\Vert _{L^{2}(\mathcal{F})\otimes\ell^{2}(\Gamma_{g})}+\frac{A}{\sqrt{r}}\right)\left(1+\left(\frac{r\log(n)}{n}\right)^{\frac{1}{b}}\right).
\end{align*}
However, $L^{2}(\mathcal{F})\otimes\ell^{2}(\Gamma_{g})$ and $L^{2}(\mathbb{H})$
are isometrically isomorphic and the operator on the right-hand side
conjugates to $P_{k_{1}}:L^{2}(\mathbb{H})\to L^{2}(\mathbb{H})$.
But, the operator norm of this is equal to $h_{1}(0)$ as the Selberg
transform is maximised for $r\in[0,\infty)$ at $r=0$. Following
the equalities and inequalities above, we thus obtain
\[
h_{1}\left(i\sqrt{\frac{1}{4}-\lambda_{1}^{\mathrm{new}}(X_{n})}\right)-h_{1}(0)\leq\mathrm{const}\cdot\left(\frac{1}{\sqrt{r}}+\left(\frac{r\log(n)}{n}\right)^{\frac{1}{b}}\right).
\]
But since $\cosh(au)-1\geq\frac{a^{2}u^{2}}{2}$ we have
\begin{align*}
h_{1}\left(i\sqrt{\frac{1}{4}-\lambda_{1}^{\mathrm{new}}(X_{n})}\right)-h_{1}(0)\\
 & \hspace*{-1.5cm}=4\sqrt{2}\int_{0}^{1}\left(\cosh\left(\sqrt{\frac{1}{4}-\lambda_{1}^{\mathrm{new}}(X_{n})}u\right)-1\right)\sqrt{\cosh(1)-\cosh(u)}\mathrm{d}u\\
 & \hspace*{-1.5cm}\geq2\sqrt{2}\left(\frac{1}{4}-\lambda_{1}^{\mathrm{new}}(X_{n})\right)\int_{0}^{1}u^{2}\sqrt{\cosh(1)-\cosh(u)}\mathrm{d}u,
\end{align*}
which means 
\[
\lambda_{1}^{\mathrm{new}}(X_{n})\geq\frac{1}{4}-\mathrm{const}\cdot\left(\frac{1}{\sqrt{r}}+\left(\frac{r\log(n)}{n}\right)^{\frac{1}{b}}\right).
\]
Choosing $r=n^{a}$ for some small $a>0$ then gives the result. 
\end{proof}

\section*{Acknowledgments}

We thank Michael Magee for very helpful discussions surrounding this
work. The initial stages of this work were carried out when all authors
were at Durham University. D.M. is funded by an Argelander Grant awarded
by the University of Bonn. J.T. is funded by the Leverhulme Trust
through a Leverhulme Early Career Fellowship (Grant No. ECF-2024-440).

\bibliographystyle{amsalpha}
\bibliography{unitarybundlesbib}

\providecommand{\bysame}{\leavevmode\hbox to3em{\hrulefill}\thinspace}
\providecommand{\MR}{\relax\ifhmode\unskip\space\fi MR }
\providecommand{\MRhref}[2]{%
  \href{http://www.ams.org/mathscinet-getitem?mr=#1}{#2}
}
\providecommand{\href}[2]{#2}
\begin{thebibliography}{GLMST21}

\bibitem[AM23]{An.Mo2023}
Nalini Anantharaman and Laura Monk, \emph{{Friedman-Ramanujan} functions in
  random hyperbolic geometry and application to spectral gaps},
  arXiv:2304.02678 (2023).

\bibitem[AM25]{An.Mo2025}
\bysame, \emph{{Friedman-Ramanujan} functions in random hyperbolic geometry and
  application to spectral gaps {II}}, arXiv:2502.12268 (2025).

\bibitem[BBD88]{BBD}
Peter Buser, Marc Burger, and Jozef Dodziuk, \emph{Riemann surfaces of large
  genus and large {$\lambda_1$}}, Geometry and analysis on manifolds
  ({K}atata/{K}yoto, 1987), Lecture Notes in Math., vol. 1339, Springer,
  Berlin, 1988, pp.~54--63. \MR{961472}

\bibitem[BC23]{Bo.Co2023}
Charles Bordenave and Benoit Collins, \emph{Norm of matrix-valued polynomials
  in random unitaries and permutations}, arXiv preprint arXiv:2304.05714
  (2023).

\bibitem[Ber16]{Be2016}
Nicolas Bergeron, \emph{The spectrum of hyperbolic surfaces}, Springer, 2016.

\bibitem[BGS84]{Bo.Gi.Sc1984}
Oriol Bohigas, Marie-Joya Giannoni, and Charles Schmit, \emph{Characterization
  of chaotic quantum spectra and universality of level fluctuation laws},
  Physical review letters \textbf{52} (1984), no.~1, 1.

\bibitem[BH99]{Br.Ha1999}
Martin~R. Bridson and Andr{\'e} Haefliger, \emph{Metric {{Spaces}} of
  {{Non-Positive Curvature}}}, Grundlehren Der Mathematischen
  {{Wissenschaften}}, vol. 319, Springer, Berlin, Heidelberg, 1999.

\bibitem[CGTv24]{Ch.Ga.Tr.va2024}
Chi-Fang Chen, Jorge {Garza-Vargas}, Joel~A. Tropp, and Ramon {van Handel},
  \emph{A new approach to strong convergence}, Annals of Mathematics, to appear
  (2024).

\bibitem[CGvH24]{Ch.Ga.Ha2024}
Chi-Fang Chen, Jorge {Garza-Vargas}, and Ramon van Handel, \emph{A new approach
  to strong convergence {{II}}. {{The}} classical ensembles}, arXiv:2412.00593
  (2024).

\bibitem[GLMST21]{Gi.Le.Sa.Th2021}
Clifford Gilmore, Etienne Le~Masson, Tuomas Sahlsten, and Joe Thomas,
  \emph{Short geodesic loops and {$L^p$} norms of eigenfunctions on large genus
  random surfaces}, Geometric and Functional Analysis \textbf{31} (2021),
  no.~1, 62--110.

\bibitem[Hid23]{Hi2023}
Will Hide, \emph{Effective lower bounds for spectra of random covers and random
  unitary bundles}, Israel Journal of Mathematics, to appear (2023).

\bibitem[HM23]{Hi.Ma2023}
Will Hide and Michael Magee, \emph{Near optimal spectral gaps for hyperbolic
  surfaces}, Annals of Mathematics \textbf{198} (2023), no.~2, 791--824.

\bibitem[HMY25]{Hu.Mc.Ya2024}
Jiaoyang Huang, Theo McKenzie, and Horng-Tzer Yau, \emph{Ramanujan property and
  edge universality of random regular graphs}, arXiv:2412.20263 (2025).

\bibitem[Hub74]{Huber}
H.~Huber, \emph{\"{U}ber den ersten {E}igenwert des {L}aplace-{O}perators auf
  kompakten {R}iemannschen {F}l\"{a}chen}, Commentarii Mathematici Helvetici
  \textbf{49} (1974), 251--259. \MR{365408}

\bibitem[LM23]{Lo.Ma2023}
Larsen Louder and Michael Magee, \emph{Strongly convergent unitary
  representations of limit groups}, arXiv:2210.08953 (2023).

\bibitem[LMS24]{Le.Sa2024}
Etienne Le~Masson and Tuomas Sahlsten, \emph{Quantum ergodicity for
  {E}isenstein series on hyperbolic surfaces of large genus}, Mathematische
  Annalen \textbf{389} (2024), no.~1, 845--898.

\bibitem[LS04]{Li.Sh2024}
M~Liebeck and A~Shalev, \emph{Fuchsian groups, coverings of riemann surfaces,
  subgroup growth, random quotients and random walks}, Journal of Algebra
  \textbf{276} (2004), no.~2, 552--601.

\bibitem[LW24]{Li.Wr2024}
Michael Lipnowski and Alex Wright, \emph{Towards optimal spectral gaps in large
  genus}, The Annals of Probability \textbf{52} (2024), no.~2, 545--575.

\bibitem[Mag24]{Ma2024}
Michael Magee, \emph{The limit points of the bass notes of arithmetic
  hyperbolic surfaces}, arXiv:2403.00928 (2024).

\bibitem[Mir13]{MirzakhaniRandom}
M.~Mirzakhani, \emph{Growth of {W}eil-{P}etersson volumes and random hyperbolic
  surfaces of large genus}, Journal of Differential Geometry \textbf{94}
  (2013), no.~2, 267--300. \MR{3080483}

\bibitem[MNP22]{Ma.Na.Pu2022}
Michael Magee, Fr{\'e}d{\'e}ric Naud, and Doron Puder, \emph{A random cover of
  a compact hyperbolic surface has relative spectral gap
  {$\frac{3}{16}-\varepsilon$}}, Geometric and Functional Analysis \textbf{32}
  (2022), no.~3, 595--661.

\bibitem[MP23]{MPasympcover}
Michael Magee and Doron Puder, \emph{The asymptotic statistics of random
  covering surfaces}, Forum of Mathematics, Pi \textbf{11} (2023), e15.

\bibitem[MPvH25]{Ma.Pu.vH2025}
Michael Magee, Doron Puder, and Ramon van Handel, \emph{Strong convergence of
  uniformly random permutation representations of surface groups},
  arXiv:2504.08988 (2025).

\bibitem[TW96]{Tr.Wi1996}
Craig~A. Tracy and Harold Widom, \emph{On orthogonal and symplectic matrix
  ensembles}, Communications in Mathematical Physics \textbf{177} (1996),
  727--754.

\bibitem[WX22]{Wu.Xu2022}
Yunhui Wu and Yuhao Xue, \emph{Random hyperbolic surfaces of large genus have
  first eigenvalues greater than {$\frac{3}{16}-\varepsilon$}}, Geometric and
  Functional Analysis \textbf{32} (2022), no.~2, 340--410.

\end{thebibliography}

\noindent Will Hide, \\
Mathematical Institute,\\
University of Oxford, \\
Andrew Wiles Building, OX2 6GG Oxford,\\
United Kingdom\\
\texttt{william.hide@maths.ox.ac.uk}~\\
\texttt{}~\\

\noindent Davide Macera, \\
Institute for Applied Mathematics\\
Faculty of Mathematics and Natural Sciences\\
Endenicher Allee 60\\
53115 Bonn\\
\texttt{macera@iam.uni-bonn.de}~\\
\texttt{}~\\

\noindent Joe Thomas, \\
Department of Mathematical Sciences,\\
Durham University, \\
Lower Mountjoy, DH1 3LE Durham,\\
United Kingdom\\
\texttt{joe.thomas@durham.ac.uk}
\end{document}